\documentclass[a4paper,leqno,10pt]{amsart}

\usepackage{epsf,graphicx,epsfig}
\usepackage{amscd}
\usepackage{amsmath,latexsym,amssymb,amsthm}
\usepackage[nospace,noadjust]{cite}
\usepackage{textcomp}
\usepackage{setspace,cite}
\usepackage{lscape,fancyhdr,fancybox}
\usepackage{textcomp}
\usepackage{bm}
\usepackage{stmaryrd}
\usepackage[all,cmtip]{xy}
\usepackage{tikz}
\usetikzlibrary{shapes,arrows,decorations.markings}
\usepackage{graphicx}

\usepackage{color}
\usepackage{url}
\usepackage{enumerate}
\usepackage[mathscr]{euscript}

  \theoremstyle{plain}

\swapnumbers
    \newtheorem{thm}{Theorem}[section]
    \newtheorem{prop}[thm]{Proposition}
   \newtheorem{lemma}[thm]{Lemma}

    \newtheorem{subsec}[thm]{}

\theoremstyle{definition}
    \newtheorem{defn}[thm]{Definition}
    \newtheorem{remark}[thm]{Remark}
\theoremstyle{remark}

\title{}
\author{}
\date{}
\usepackage{amssymb}

\usepackage{hyperref}
\hypersetup{
	colorlinks,
	citecolor=blue,
	filecolor=black,
	linkcolor=blue,
	urlcolor=black
}

\title{Cohomology and deformations of dendriform coalgebras}
\author{Apurba Das}
\address{Department of Mathematics and Statistics,
Indian Institute of Technology, Kanpur 208016, Uttar Pradesh, India.}
\email{apurbadas348@gmail.com}

\subjclass[2010]{16T15, 17A30, 16E40, 18G55}
\keywords{Dendriform (co)algebra, (co)Hochschild cohomology, Formal deformation, Homotopy (co)algebra}

\begin{document}

\begin{abstract}
Dendriform coalgebras are the dual notion of dendriform algebras and are splitting of associative coalgebras. In this paper, we define a cohomology theory for dendriform coalgebras based on some combinatorial maps. We show that the cohomology with self coefficients governs the formal deformation of the structure. We also relate this cohomology with the cohomology of dendriform algebras, coHochschild (Cartier) cohomology of associative coalgebras and cohomology of Rota-Baxter coalgebras which we introduce in this paper. Finally, using those combinatorial maps, we introduce homotopy analogue of dendriform coalgebras and study some of their properties.
\end{abstract}

\maketitle

\thispagestyle{empty}

\section{Introduction}
Dendriform algebras were first introduced by Jean-Louis Loday in his study on periodicity phenomenons in algebraic $K$-theory \cite{loday}. These algebras are Koszul dual to diassociative algebras (also called associative dialgebras). More precisely, a dendriform algebra is a vector space equipped with two binary operations satisfying three new identities. The sum of the two operations turns out to be associative. Thus, a dendriform algebra can be thought of like a splitting of associative algebras. Dendriform algebras are closely related to Rota-Baxter algebras \cite{aguiar, ebr-guo}. Recently, the present author defines a cohomology and deformation theory for dendriform algebras \cite{das1}. See \cite{das-saha,fard-et, ebr-guo, ronco} and references therein for more literature about dendriform algebras.

\medskip

The dual picture of dendriform algebras is given by dendriform coalgebras. It is given by two coproducts on a vector space satisfying three identities (dual to dendriform algebra identities), and the sum of the two coproducts makes the underlying vector space into an associative coalgebra. Such algebras arise naturally from Rota-Baxter coalgebras as defined in \cite{jian-zhang}.

\medskip

In section \ref{sec-3}, we introduce representations and cohomology of dendriform coalgebras. This cohomology is based on certain combinatorial maps as defined in \cite{das1}.
The cohomology with coefficients in itself inherits a Gerstenhaber algebra structure.
 Like coHochschild (Cartier) cohomology of associative coalgebras are dual (in an appropriate sense) to Hochschild cohomology of associative algebras, the cohomology for dendriform coalgebras are dual to the cohomology of dendriform algebras as defined in \cite{das1}. We also show that there is a morphism from the cohomology of a dendriform coalgebra to the coHochschild cohomology of the corresponding associative coalgebra. We also introduce the cohomology of a Rota-Baxter operator and relate with the cohomology of the induced dendriform coalgebra.

\medskip

Formal deformation theory of associative algebras was first introduced by Gerstenhaber \cite{gers} and then extend to several other algebraic structures, such as Lie algebras, Leibniz algebras and many others \cite{nij-ric, bala, yau, das2 }. In \cite{gers-sch}, Gerstenhaber and Schack defined the same theory for associative coalgebras.
In section \ref{sec-4}, we apply the Gerstenhaber's approach of the formal deformation theory to dendriform coalgebras. The vanishing of the second cohomology of a dendriform coalgebra (with coefficients in itself) implies the rigidity of the structure. The vanishing of the third cohomology ensures that a finite order deformation can be extended to deformation of next order.

\medskip

Finally, in section \ref{sec-5}, we introduce a homotopy analogue of dendriform coalgebras, also called $\text{Dend}_\infty$-coalgebras. We show that they are splitting of $A_\infty$-coalgebras, thus, generalizing a result from the non-homotopic case. We also defined Rota-Baxter $A_\infty$-coalgebras (a  homotopy analogue of Rota-Baxter coalgebras) and show that they induce $\text{Dend}_\infty$-coalgebra structures. These results are the coalgebraic version of the results obtained in \cite{das1}.

\medskip

In Appendix (section \ref{sec-6}), we recall the coHochschild cohomology of associative coalgebras and some basics on non-symmetric operad with multiplication which have been used in the main body of the paper. All vector spaces, linear maps, tensor products are over a field $\mathbb{K}$ of characteristic zero unless stated otherwise.

\section{Dendriform (co)algebras}\label{sec-2}

In this section, we recall coHochschild cohomology of associative coalgebras and dendriform structures (algebras and coalgebras). See \cite{aguiar, das1, jian-zhang, loday, ma-liu} for more details.


\medskip

Let $(C, \triangle)$ be an associative coalgebra, i.e. $\triangle : C \rightarrow C \otimes C$ is a linear map satisfying
\begin{align*}
(\triangle \otimes \text{id}) \circ \triangle =~& (\text{id} \otimes \triangle) \circ \triangle.
\end{align*}
A bicomodule over it consists of a vector space $M$ together with linear maps $\triangle^l : M \rightarrow C \otimes M$ and $\triangle^r : M \rightarrow M \otimes C$ (called left and right coactions) satisfying
\begin{align*}
(\triangle \otimes \text{id}) \circ \triangle^l = (\text{id} \otimes \triangle^l) \circ \triangle^l, \quad
 (\triangle^l \otimes \text{id}) \circ \triangle^r = (\text{id} \otimes \triangle^r) \circ \triangle^l, \quad
 (\triangle^r \otimes \text{id}) \circ \triangle^r = (\text{id} \otimes \triangle) \circ \triangle^r.
\end{align*}

Given a bicomodule $(M, \triangle^l, \triangle^r)$ over an associative coalgebra $(C, \triangle)$, the coHochschild cohomology is given by the cohomology of the following cochain complexe. For each $n \geq 0$, the $n$-th cochain group $C_{\text{coHoch}}^n (M, C)$ is given by $C_{\text{coHoch}}^0 (M, C) = 0$ and $C_{\text{coHoch}}^n (M, C) = \text{Hom} (M, C^{\otimes n})$, for $n \geq 1$. The coboundary $\delta_c : C_{\text{coHoch}}^n (M, C) \rightarrow C_{\text{coHoch}}^{n+1} (M, C)$ is given by
\begin{align}
\delta_c \sigma = (\text{id} \otimes  \sigma) \circ \triangle^l ~+~ \sum_{i=1}^n (-1)^i~(\text{id}_{C^{\otimes (i-1)}} \otimes \triangle \otimes \text{id}_{C^{\otimes (n-i)}}) \circ \sigma ~+~ (-1)^{n+1} (\sigma \otimes \text{id}) \circ \triangle^r , 
\end{align}
for $\sigma \in  C_{\text{coHoch}}^n (M, C).$\\

\noindent {\em Dendriform coalgebras.}

\medskip

\begin{defn}
A dendriform algebra $(A, \prec, \succ)$ consists of a vector space $A$ together with linear maps $\prec, \succ : A \otimes A \rightarrow A$ satisfying the following identities:
\begin{align*}
(a \prec b) \prec c =~& a \prec (b \prec c + b \succ c),\\
(a \succ b) \prec c =~& a \succ (b \prec c),\\
(a \prec b + a \succ b) \succ c =~& a \succ (b \succ c), ~~\text{for all }a , b, c \in A.
\end{align*}
\end{defn}

By dualizing the notion of dendriform algebra, one obtains the following notion.
\begin{defn}
A dendriform coalgebra $(C, \triangle_\prec, \triangle_\succ)$ consists of a vector space $C$ together with linear maps $\triangle_\prec, \triangle_\succ : C \rightarrow C \otimes C$ that satisfies
\begin{align}
(\triangle_\prec \otimes  \text{id}) \circ \triangle_\prec =~& (\text{id} \otimes (\triangle_\prec + \triangle_\succ)) \circ \triangle_\prec, \label{c1}\\
(\triangle_\succ \otimes \text{id}) \circ \triangle_\prec =~& (\text{id} \otimes \triangle_\prec ) \circ \triangle_\succ, \label{c2}\\
((\triangle_\prec + \triangle_\succ ) \otimes \text{id} ) \circ \triangle_\succ =~&  (\text{id} \otimes \triangle_\succ ) \circ \triangle_\succ . \label{c3}
\end{align}
\end{defn}

It follows from (\ref{c1}) + (\ref{c2}) + (\ref{c3}) that the sum operation $\triangle = \triangle_\prec + \triangle_\succ$ satisfies
$(\triangle \otimes \text{id}) \circ \triangle = (\text{id} \otimes \triangle ) \circ \triangle,$
hence, an associative coproduct. Thus, dendriform coalgebras can be thought of as splitting of associative coalgebras.

Any associative coalgebra $(C, \triangle)$ can be thought of as a dendriform coalgebra with $\triangle_\prec = \triangle$ and $\triangle_\succ = 0$ (alternatively, $\triangle_\prec = 0$ and $\triangle_\succ = \triangle$).

\begin{defn}
Let $(C, \triangle)$ be an associative coalgebra and $(M, \triangle^l, \triangle^r)$ be a bicomodule. 
A linear map $ T : C \rightarrow M$ is said to be a relative Rota-Baxter operator if $T$ satisfies
\begin{align*}
(T\otimes T ) \circ \triangle = (\text{id} \otimes T ) \circ \triangle^r \circ T + (T \otimes \text{id}) \circ \triangle^l \circ T.
\end{align*}
\end{defn}
A Rota-Baxter operator is a relative Rota-Baxter operator for the self bicomodule.

\begin{prop}\label{rota-dend}
Let $T : C \rightarrow M$ be a relative Rota-Baxter coalgebra. Then $(M, \triangle_\prec, \triangle_\succ)$ is a dendriform coalgebra where $\triangle_\prec = (\mathrm{id} \otimes T) \circ \triangle^r $  ~and~ $\triangle_\succ = (T \otimes \mathrm{id} ) \circ \triangle^l.$
\end{prop}

\begin{proof}
It follows from the definition of relative Rota-Baxter operator that
\begin{align*}
(T \otimes T) \circ \triangle = (\triangle_\prec + \triangle_\succ) \circ T.
\end{align*}
Next observe that
\begin{align*}
(\text{id} \otimes T \otimes T ) \circ (\triangle^r \otimes \text{id}) \circ \triangle^r 
=~& ((\text{id} \otimes T) \triangle^r  \otimes T ) \circ \triangle^r
= (\triangle_\prec \otimes T ) \circ \triangle^r \\
=~& (\triangle_\prec \otimes \text{id}) \circ (\text{id} \otimes T) \circ \triangle^r = (\triangle_\prec \otimes \text{id} ) \circ \triangle_\prec .
\end{align*}
On the other hand
\begin{align*}
(\text{id} \otimes T \otimes T) \circ (\text{id} \otimes \triangle) \circ \triangle^r 
=~& (\text{id} \otimes (T \otimes T) \circ \triangle ) \circ \triangle^r 
= (\text{id} \otimes (\triangle_\prec + \triangle_\succ) \circ T ) \circ \triangle^r \\
=~& (\text{id} \otimes (\triangle_\prec + \triangle_\succ)) \circ (\text{id} \otimes T ) \circ \triangle^r 
= (\text{id} \otimes (\triangle_\prec + \triangle_\succ)) \circ \triangle_\prec.
\end{align*}
Since $\triangle$ is an associative coproduct, it follows that 
\begin{align*}
 (\triangle_\prec \otimes \text{id} ) \circ \triangle_\prec = (\text{id} \otimes (\triangle_\prec + \triangle_\succ)) \circ \triangle_\prec. 
\end{align*} 
 Similarly, by evaluating both sides of the identity $(T \otimes \text{id} \otimes T ) \circ (\triangle^l \otimes \text{id}) \circ \triangle^r = (T \otimes \text{id} \otimes T ) \circ (\text{id} \otimes \triangle^r) \circ \triangle^l$, we obtain $(\triangle_\succ \otimes \text{id}) \circ \triangle_\prec = (\text{id} \otimes \triangle_\prec ) \circ \triangle_\succ$. Finally, evaluating both sides of the identity $(T \otimes T \otimes \text{id}) \circ (\triangle \otimes \text{id}) \circ \triangle^l = (T \otimes T \otimes \text{id}) \circ (\text{id} \otimes \triangle^l) \circ \triangle^l$ will gives rise to $((\triangle_\prec + \triangle_\succ ) \otimes \text{id} ) \circ \triangle_\succ =  (\text{id} \otimes \triangle_\succ ) \circ \triangle_\succ$.
Hence the proof.
\end{proof}


\section{Cohomology of dendriform coalgebras} \label{sec-3}
In this section, we introduce representations and cohomology of dendriform coalgebras. This can be seen as a dual to the cohomology theory of dendriform algebras as explained in \cite{das1}. We also relate this cohomology with the coHochschild cohomology of the corresponding associative coalgebra.

Let $C = (C, \triangle_\prec, \triangle_\succ)$ be a dendriform coalgebra.
\begin{defn}
A bicomodule over $C$ consists of a vector space $M$ together with linear maps
\begin{align*}
\triangle^l_\prec , \triangle^l_\succ : M \rightarrow C \otimes M ~~ \qquad \text{ and } \qquad \qquad ~~ ~~ \triangle^r_\prec , \triangle^r_\succ : M \rightarrow M \otimes C
\end{align*}
satisfying the following $9$ identities
\begin{align}
(\triangle_\prec \otimes  \text{id}) \circ \triangle^l_\prec =~& (\text{id} \otimes (\triangle_\prec^l + \triangle_\succ^l)) \circ \triangle_\prec^l, \label{r1}\\
(\triangle_\succ \otimes \text{id}) \circ \triangle_\prec^l =~& (\text{id} \otimes \triangle_\prec^l ) \circ \triangle_\succ^l, \label{r2}\\
((\triangle_\prec + \triangle_\succ ) \otimes \text{id} ) \circ \triangle_\succ^l =~&  (\text{id} \otimes \triangle_\succ^l ) \circ \triangle_\succ^l, \label{r3}\\
(\triangle_\prec^l \otimes  \text{id}) \circ \triangle_\prec^r =~& (\text{id} \otimes (\triangle_\prec^r + \triangle_\succ^r)) \circ \triangle_\prec^l, \label{r4}\\
(\triangle_\succ^l \otimes \text{id}) \circ \triangle_\prec^r =~& (\text{id} \otimes \triangle_\prec^r ) \circ \triangle_\succ^l, \label{r5}\\
((\triangle_\prec^l + \triangle_\succ^l ) \otimes \text{id} ) \circ \triangle_\succ^r =~&  (\text{id} \otimes \triangle_\succ^r ) \circ \triangle_\succ^l , \label{r6}\\
(\triangle_\prec^r \otimes  \text{id}) \circ \triangle_\prec^r =~& (\text{id} \otimes (\triangle_\prec + \triangle_\succ)) \circ \triangle_\prec^r, \label{r7}\\
(\triangle_\succ^r \otimes \text{id}) \circ \triangle_\prec^r =~& (\text{id} \otimes \triangle_\prec ) \circ \triangle_\succ^r, \label{r8}\\
((\triangle_\prec^r + \triangle_\succ^r ) \otimes \text{id} ) \circ \triangle_\succ^r =~&  (\text{id} \otimes \triangle_\succ ) \circ \triangle_\succ^r . \label{r9}
\end{align}
\end{defn}

It follows from the above definition that $C$ is a bicomodule over itself with $\triangle^l_\prec  = \triangle^r_\prec = \triangle_\prec$ and $\triangle^l_\succ = \triangle^r_\succ = \triangle_\succ.$\\

\begin{prop} (Semi-direct product)
Let $C$ be a dendriform coalgebra and $M$ be a bicomodule over $C$. Then the direct sum $C \oplus M$ inherits a dendriform coalgebra structure with comultiplications $\triangle_\prec, ~ \triangle_\succ : (C \oplus M) \rightarrow (C \oplus M) \otimes (C \oplus M)$ are given by (on generators)
\begin{align*}
\triangle_\prec (c, 0) := \triangle_\prec (c), \quad \triangle_\prec( 0, m) := \triangle^l_\prec (m) + \triangle_\prec^r (m), \\
\triangle_\succ (c, 0) := \triangle_\succ (c), \quad \triangle_\succ( 0, m) := \triangle^l_\succ (m) + \triangle_\succ^r (m).\\
\end{align*}
\end{prop}

For convenience, we define maps $\triangle^l \in \text{Hom} ( \mathbb{K}[C_2] \otimes M, C \otimes M)$ and $\triangle^r \in \text{hom} (\mathbb{K}[C_2] \otimes M, M \otimes C)$ by
\begin{align*}
\triangle^l ([1]; ~\_~ ) = \triangle_\prec^l, \quad \triangle^l ([2]; ~\_~ ) = \triangle^l_\succ, \quad \triangle^r ([1]; ~\_~ ) = \triangle^r_\prec \quad \text{and} \quad \triangle^r([2]; ~\_~) = \triangle^r_\succ.
\end{align*}
These notations will be used in the coboundary operator defining the cohomology.

To define the cohomology of a dendriform coalgebra, we need to recall certain combinatorial maps which are defined in \cite{das1}. Let $C_n$ be the set of first $n$ natural numbers. For convenience, we denote the elements of $C_n$ as $\{ [1], [2], \ldots, [n] \}$. For each $m, n \geq 1$, there are maps $R_0 (m; \overbrace{1, \ldots, 1, \underbrace{n}_{i\text{-th place}}, 1, \ldots, 1}^{m}) : C_{m+n-1} \rightarrow C_m$ given by
\begin{align*} R_0 (m; 1, \ldots, 1, n, 1, \ldots, 1) ([r]) ~=~
\begin{cases} [r] ~~~ &\text{ if } ~~ r \leq i-1 \\ [i] ~~~ &\text{ if } i \leq r \leq i +n -1 \\
[r -n + 1] ~~~ &\text{ if } i +n \leq r \leq m+n -1. \end{cases}
\end{align*}
There are also maps $R_i (m; \overbrace{1, \ldots, 1, \underbrace{n}_{i\text{-th place}}, 1, \ldots, 1}^{m}) : C_{m+n-1} \rightarrow \mathbb{K}[C_n]$ defined by
\begin{align*} R_i (m; 1, \ldots, 1, n, 1, \ldots, 1) ([r]) ~=~
\begin{cases} [1] + [2] + \cdots + [n] ~~~ &\text{ if } ~~ r \leq i-1 \\ [r - (i-1)] ~~~ &\text{ if } i \leq r \leq i +n -1 \\
[1]+ [2] + \cdots + [n] ~~~ &\text{ if } i +n \leq r \leq m+n -1. \end{cases}\\
\end{align*}

For any vector space $C$, we consider a collection of vector spaces $\{ \mathcal{O}(n) |~ n \geq 1 \}$ by $\mathcal{O}(n) = \text{Hom} (\mathbb{K}[C_n] \otimes C, C^{\otimes n} )$, for $n \geq 1$. Define partial compositions $\bullet_i : \mathcal{O}(m) \otimes \mathcal{O}(n) \rightarrow \mathcal{O}(m+n-1)$ by
\begin{align*}
& (f \bullet_i g) ([r]; ~\_~) \\
~& = (\text{id}^{\otimes (i-1)} \otimes g (R_i (m; 1, \ldots, n, \ldots, 1)[r]; ~\_~) \otimes \text{id}^{\otimes (m-i)}) \circ f (R_0 (m; 1, \ldots, n, \ldots, 1)[r]; ~\_~),
\end{align*}
for $[r] \in C_{[m+n-1]}.$ Then we have the following.

\begin{prop}
With the above notations, the collection of vector spaces $\{ \mathcal{O}(n) |~ n \geq 1 \}$ forms a non-symmetric operad with partial compositions $\bullet_i$ and identity element $\mathrm{id} \in \mathcal{O}(1)$ being given by $\mathrm{id } ([1]; ~\_~) = \mathrm{id}.$
\end{prop}

This result (and its proof) is completely dual to Proposition 2.2 \cite{das1}. In fact, this proposition was our initial motivation to study cohomology of dendriform coalgebras which we will define below.

Next assume that $(C, \triangle_\prec, \triangle_\succ)$ is a dendriform coalgebra. Define an element $\triangle \in \mathcal{O}(2) = \text{Hom} (\mathbb{K}[C_2] \otimes C, C^{\otimes 2} )$ by
\begin{align*}
\triangle ([1]; ~\_~ ) = \triangle_\prec     ~~~~ \text{ and } ~~~~  \triangle ([2]; ~\_~ ) = \triangle_\succ.
\end{align*}
Then it is easy to see that $\triangle \in \mathcal{O}(2)$ defines a multiplication in the operad.


We are now in a position to define the cohomology of a dendriform coalgebra with coefficients in a bicomodule. Let $M$ be a bicomodule over a dendriform coalgebra $C$. We define the group of $n$-cochains as $C^n_{\text{coDend}}(M, C) := \text{Hom} (\mathbb{K}[C_n] \otimes M , C^{\otimes n})$, for $n \geq 1$, and $C^0_{\text{coDend}}(M, C) = 0$. The coboundary map $\delta_c : C^n_{\text{coDend}}(M, C) \rightarrow C^{n+1}_{\text{coDend}}(M, C)$ is defined by\\
\begin{align}\label{dend-co-diff}
( \delta_c \sigma )& ([r] \otimes m) \\
 =~& (\text{id} \otimes \sigma (R_2 (2;1,n)[r];~ \_~ )) \circ \triangle^l (R_0 (2;1, n)[r]; m) \nonumber \\
+~& \sum_{i=1}^n  (-1)^i ~(\text{id}_{C^{\otimes (i-1)}} \otimes \triangle_{R_i (n; 1, \ldots, 2, \ldots, 1)[r]} \otimes \text{id}_{C^{\otimes (n-i)}}) \circ \sigma (R_0 (n;1, \ldots, 2, \ldots, 1)[r]; m) \nonumber \\
+~& (-1)^{n+1}~ (\sigma (R_1 (2;n,1)[r]; ~\_ ~) \otimes \text{id} ) \circ  \triangle^r (R_0 (2;n,1)[r]; m), \nonumber
\end{align}
for $\sigma \in C^n_{\text{coDend}}(M, C)$, $[r] \in C_{n+1}$ and $m \in M$.

Then we have $(\delta_c)^2 = 0$ (see the Remark below) and the corresponding cohomology groups are denoted by $H^n_{\text{coDend}}(M, C)$, for $n \geq 1$.


\begin{remark}
When $M = C$ with the obvious bicomodule given by $\triangle_\prec^l = \triangle_\prec^r = \triangle_\prec$ and $\triangle_\succ^l = \triangle_\succ^r = \triangle_\succ$, the above coboundary
map coincides with the one induced from the multiplication on the operad. Therefore, the cohomology inherits
a Gerstenhaber algebra structure. This first observation also ensures that $(\delta_c)^2 = 0$ for the
coboundary map (\ref{dend-co-diff}) with coefficients in an arbitrary bicomodule.
\end{remark}

\begin{remark}Note that the element $\triangle \in \text{Hom} (\mathbb{K}[C_2] \otimes C, C^{\otimes 2})$ defines a $2$-cocycle in the cohomology complex of $C$ with coefficients in itself (which follows from dendriform coalgebra identities) as
\begin{align*}
(\delta_c \triangle ) ([1] \otimes m) =~& 2 \big[ (\text{id} \otimes (\triangle_\prec + \triangle_\succ)) \circ \triangle_\prec -  (\triangle_\prec \otimes  \text{id}) \circ \triangle_\prec  \big] = 0,\\
(\delta_c \triangle ) ([2] \otimes m) =~& 2 \big[ (\text{id} \otimes \triangle_\prec ) \circ \triangle_\succ -  (\triangle_\succ \otimes \text{id}) \circ \triangle_\prec   \big] = 0,\\
(\delta_c \triangle ) ([3] \otimes m) =~& 2 \big[ (\text{id} \otimes \triangle_\succ ) \circ \triangle_\succ -  ((\triangle_\prec + \triangle_\succ ) \otimes \text{id} ) \circ \triangle_\succ  \big] = 0.
\end{align*}
The $2$-cocycle $\triangle$ is infact a coboundary and is given by $\triangle = \delta_c (\text{id}).$ To show this, we observe that
\begin{align*}
\delta_c (\text{id}) ([1] ; m) =~& (\text{id} \otimes \text{id}) \circ \triangle_\prec ~-~ \triangle_\prec ~+ ~ (\text{id} \otimes \text{id}) \circ \triangle_\prec = \triangle_\prec, \\
\delta_c (\text{id}) ([2] ; m) =~& (\text{id} \otimes \text{id}) \circ \triangle_\succ ~-~ \triangle_\succ ~+ ~ (\text{id} \otimes \text{id}) \circ \triangle_\succ = \triangle_\succ.\\
\end{align*}
\end{remark}

Let $(A, \prec, \succ)$ be a finite dimensional dendriform algebra. Then the dual vector space $A^*$ equipped with the coproducts
\begin{align*}
\triangle_\prec = w^{-1} \circ \prec^*   ~~~~ \text{   and   }   ~~~~   \triangle_\succ = w^{-1} \circ \succ^*
\end{align*}
forms a dendriform coalgebra. Here $w : (A^*)^{\otimes 2} \rightarrow (A^{\otimes 2})^{*}$ denote the isomorphism given by $w (f \otimes g) (a \otimes b) = f(a) g(b).$ In the following, we show that the dendriform algebra cohomology of $A$ is isomorphic to the dendriform coalgebra cohomology of $A^*$.

Let us first recall the cohomology of a dendriform algebra with coefficients in itself. It has been shown in \cite{das1} that for any vector space $A$, the collection of spaces $\{ \text{Hom} (\mathbb{K}[C_n] \otimes A^{\otimes n}, A ) \}_{n \geq 1}$ inherits an operad structure with partial compositions given by
\medskip
\begin{align*}
&(f \circ_i g) ([r]; a_1, \ldots, a_{m+n-1}) \\
&= f (R_0 (m; 1, \ldots, n, \ldots, 1)[r];~ a_1, \ldots, a_{i-1}, g (R_i (m; 1, \ldots, n, \ldots, 1)[r]; a_i, \ldots, a_{i+n-1}), a_{i+n}, \ldots, a_{m+n-1})
\end{align*}

\medskip

\noindent for $ f \in \text{Hom} (\mathbb{K}[C_m] \otimes A^{\otimes m}, A ), ~ g \in \text{Hom} (\mathbb{K}[C_n] \otimes A^{\otimes n}, A ), ~[r] \in C_{m+n-1}$ and $a_1, \ldots, a_{m+n-1} \in A$. If $(A, \prec, \succ)$ is a dendriform algebra, it defines a multiplication $\pi \in \text{Hom} (\mathbb{K}[C_2] \otimes A^{\otimes 2}, A )$ on the above operad, where $\pi$ is given by $\pi ([1]; ~) = ~\prec$ and $\pi ([2]; ~ ) =~ \succ$. The differential induced from this multiplication is precisely the coboundary operator defining the cohomology of $A$. See \cite{das1} for more details and explicit formula for the coboundary.

Note that if $A$ is finite dimensional, then the isomorphism $w$ can be generalized to isomorphisms (also denoted by the same notation is no confusion arises) $w : (A^*)^{\otimes n} \rightarrow (A^{\otimes n})^{*}$, for $n \geq 1$, given by
\begin{align*}
w (f_1 \otimes \cdots \otimes f_n) (a_1 \otimes \cdots \otimes a_n) = f_1(a_1) \cdots f_n (a_n).
\end{align*}

\begin{prop}
For any finite dimensional vector space $A$, there is an isomorphism (in the category of vector spaces) of operads
\begin{align*}
\big\{ \mathrm{Hom}(\mathbb{K}[C_n] \otimes A^{\otimes n}, A ), \circ_i \big\}_{n \geq 1} ~& \rightarrow  \big\{ \mathrm{Hom} (\mathbb{K}[C_n] \otimes A^*, {A^*}^{\otimes n}), \bullet_i \big\}_{n \geq 1} \\
f := (f_{[1]}, \ldots, f_{[n]} ) ~& \mapsto w^{-1} \circ f^* := (w^{-1} \circ f_{[1]}^*, \ldots, w^{-1} \circ f_{[n]}^*).
\end{align*}
Moreover, if $A$ is a dendriform algebra with the dual dendriform coalgebra structure on $A^*$, then the above morphism of operads preserve the corresponding multiplications (in the operads).
\end{prop}

\begin{proof}
The second statement is clear if we prove the first one. To prove the first one, we only need to prove that
\begin{align*}
w^{-1} \circ (f \circ_i g)^*_{[r]} = ((w^{-1} \circ f^*) \bullet_i (w^{-1} \circ g^*)) ([r]; ~\_~),
\end{align*}
for $f \in \mathrm{Hom}(\mathbb{K}[C_m] \otimes A^{\otimes m}, A )$, $g \in \mathrm{Hom}(\mathbb{K}[C_n] \otimes A^{\otimes n}, A )$, $1 \leq i \leq m$ and $[r] \in C_{m+n-1}$. This follows from the definitions of partial compositions $\circ_i$ and $\bullet_i$.
\end{proof}

As a corollary, we get the following.
\begin{thm}
If $A$ is a finite-dimensional dendriform algebra, then the cohomology of $A$ and the dendriform coalgebra cohomology of $A^*$ are isomorphic.
\end{thm}

Let $M$ be a bicomodule over a dendriform coalgebra $(C, \triangle_\prec, \triangle_\succ)$. Then $M$ is also a bicomodule over the associative coalgebra $(C, \triangle = \triangle_\prec + \triangle_\succ )$ with left and right coactions
\begin{align*}
\triangle^l = \triangle^l_\prec + \triangle^l_\succ ~~~ \text{ and } ~~~ \triangle^r = \triangle^r_\prec + \triangle^r_\succ.
\end{align*}

In the next theorem, we relate the cohomology of a dendriform coalgebra with the coHochschild cohomology of the corresponding associative coalgebra. See the Appendix for the coHochschild cohomology of associative coalgebras.

\begin{thm}
The map $S : C^n_{\mathrm{coDend}} (M, C) \rightarrow C^n_{\mathrm{coHoch}} (M, C) $ defined by
\begin{align*}
S(f) := f_{[1]} + \cdots + f_{[n]}
\end{align*}
commute with the differentials, that is, $\delta_{\mathrm{coHoch}} \circ S = S \circ \delta_{\mathrm{coDend}}$. Thus, it induces a map $S_* : H^n_{\mathrm{coDend}} (M, C) \rightarrow H^n_{\mathrm{coHoch}} (M, C).$ Moreover, when $M = C$, the induced map on cohomology is a morphism between Gerstenhaber algebras.
\end{thm}

\begin{proof}
The first statement is a straightforward calculation.

To prove the second part, we consider the coendomorphism operad $\text{coEnd}_C$ associated to the vector space $C$. Similar to the first part, one can show that the map
\begin{align*}
S : \mathrm{Hom} ( \mathbb{K}[C_n] \otimes C, C^{\otimes n}) ~& \rightarrow \text{coEnd}_C (n) = \text{Hom} ( C, C^{\otimes n}), \\
f ~& \mapsto f ([1], ~\_~) + \cdots + f ([n], ~ \_~ )
\end{align*}
is a morphism between operads. The dendriform structure $(\triangle_\prec, \triangle_\succ)$ on $C$ defines a multiplication on the left hand side operad whereas the associative coalgebra structure $\triangle_\prec + \triangle_\succ$ on $C$ defines a multiplication on the right hand side operad. The map $S$ preserve the corresponding multiplications. Therefore, $S$ will induces a map between Gerstenhaber algebras.
\end{proof}

\subsection{Relation with the cohomology of Rota-Baxter operators} In \cite{das3} the author introduced the cohomology of a relative Rota-Baxter operator on an associative algebra and relates with the cohomology of dendriform algebras. In this subsection, we find the dual version of the results of \cite{das3} by introducing the cohomology of a relative Rota-Baxter operator on a coalgebra.

Let $(C, \triangle)$ be an associative coalgebra and $(M, \triangle^l, \triangle^r)$ be a bicomodule.  Consider the vector space $A = C \oplus M$ and the graded Lie algebra on $\oplus_{n \geq 0} \mathrm{Hom} ( A, A^{\otimes n+1})$ induced from the coendomorphism operad $\mathrm{coEnd}_A$. We observe that $\triangle, \triangle^l , \triangle^r \in \mathrm{Hom}( A, A^{\otimes 2})$.

\begin{prop}
With the above notations, the sum $\triangle + \triangle^l + \triangle^r \in \mathrm{Hom}( A, A^{\otimes 2})$ is a Maurer-Cartan element in the graded Lie algebra on $\oplus_{n \geq 0} \mathrm{Hom} ( A, A^{\otimes n+1})$.
\end{prop}

This result is standard in the philosophy that a (co)algebra and a representation can be simultaneously defined by a Maurer-Cartan element in a suitably graded Lie algebra associated to the direct sum vector space. See \cite[Proposition 2.11]{das3} for instance.

This shows that $\triangle + \triangle^l + \triangle^r$ induces a differential 
\begin{align*}
d_{ \triangle + \triangle^l + \triangle^r } := [ \triangle + \triangle^l + \triangle^r,~~  \_ ~~]
\end{align*}
that makes $   (\oplus_{n \geq 0} \mathrm{Hom} (A, A^{\otimes n+1} , [~,~], d_{   \triangle + \triangle^l + \triangle^r })$ into a differential  graded Lie algebra. It is easy to see that the graded subspace $\oplus_{n \geq 0} \mathrm{Hom}(C, M^{\otimes n+1})$ is an abelian subalgebra. Therefore, by the derived bracket construction of Voronov \cite{voro} yields a graded Lie bracket on $\oplus_{n \geq 1} \mathrm{Hom} (C, M^{\otimes n})$ by 
\begin{align*}
\llbracket P, Q \rrbracket := (-1)^m ~ [ d_{ \triangle + \triangle^l + \triangle^r } (P), Q ],
\end{align*}
for $P \in \mathrm{Hom}(C, M^{\otimes m})$ and $Q \in  \mathrm{Hom}(C, M^{\otimes n})$. In particular, for $P \in \mathrm{Hom}(C, M)$ and $Q \in  \mathrm{Hom}(C, M^{\otimes n})$, the bracket is explicitly given by
\begin{align}\label{derived-brkt}
\llbracket P, Q \rrbracket =~& (Q \otimes \mathrm{id}) \circ \triangle^l \circ P - (-1)^n (\mathrm{id} \otimes Q) \circ \triangle^r \circ P \\
&- (-1)^n \big\{ \sum_{i=1}^n (-1)^i~ (\mathrm{id}_{M^{\otimes i-1}} \otimes \otimes ((\mathrm{id} \otimes P) \circ \triangle^r + ( P \otimes \mathrm{id}) \circ \triangle^l) \otimes \mathrm{id}_{M^{\otimes n-i}} ) \circ Q \nonumber \\
&+ (-1)^n (\mathrm{id} \otimes Q) \circ (P \otimes \mathrm{id} ) \circ \triangle - ( Q \otimes \mathrm{id}) \circ (\mathrm{id} \otimes P) \circ \triangle. \nonumber
\end{align}

Thus, it follows from (\ref{derived-brkt}) that a linear map $T: C \rightarrow M$ is a relative Rota-Baxter operator if and only if $\llbracket T, T \rrbracket = 0$, i.e. $T$ is a Maurer-Cartan element in the above graded Lie algebra. Therefore, a relative Rota-Baxter operator $T$ induces a differential $d_T := \llbracket T, ~ \rrbracket$ on the graded vector space $\oplus_{n \geq 1} \mathrm{Hom}(C, M^{\otimes n})$. The corresponding cohomology groups which are denoted by $H^\bullet_T (C, M)$ are called the cohomology of the relative Rota-Baxter operator $T$. Similar to \cite{das3} one can show that such cohomology governs the deformation of the relative Rota-Baxter operator $T$.

In the following, we will show that the above cohomology for a relative Rota-Baxter operator $T$ can be seen as the coHochschild cohomology of an associative coalgebra. We start with the following lemma.

\begin{lemma}
Let $T : C \rightarrow M$ be a relative Rota-Baxter operator. Then $M$ is an associative coalgebra with the coproduct $\triangle_* : M \rightarrow M \otimes M$ given by $\triangle_* = (\mathrm{id} \otimes T) \circ \triangle^r + (T \otimes \mathrm{id}) \circ \triangle^l$. Moreover, $C$ is a bicomodule of $M$ with left and right coactions $\triangle^l_* , \triangle^r_*$ given by
\begin{align*}
\triangle^l_* = (T \otimes \mathrm{id}) \circ \triangle - \triangle^r \circ T \qquad
\triangle^r_* = (\mathrm{id} \otimes T) \circ \triangle - \triangle^l \circ T.
\end{align*}
\end{lemma}

\begin{proof}
The first part follows from the fact $(M, \triangle_\prec, \triangle_\succ )$ is a dendriform coalgebra and $\triangle_* = \triangle_\prec + \triangle_\succ$. The second part is dual to Lemma \cite[Lemma 3.1]{das3}. Hence we do not repeat it here.
\end{proof}

 It follows from the above lemma that we can consider the coHochschild cohomology of the associative coalgebra $M$ with coefficients in the bicomodule $C$. More precisely, we define $C^n_{\mathrm{coHoch}}( C, M) = \mathrm{Hom} (C, M^{\otimes n})$ for $n \geq 1$ and the differential $\delta_{\mathrm{coHoch}} : C^n_{\mathrm{coHoch}}( C, M) \rightarrow C^{n+1}_{\mathrm{coHoch}}( C, M) $ given by
 \begin{align}\label{cohoch-rota}
(\delta_{\mathrm{coHoch}} f) =~& (\mathrm{id} \otimes f) \circ (T \otimes \mathrm{id}) \circ \triangle  ~ - ~ (\mathrm{id} \otimes f) \circ \triangle^r \circ T \\
&+ \sum_{i=1}^n (-1)^i ~ \big(  \mathrm{id}_{M^{\otimes i-1}} \otimes ((\mathrm{id} \otimes T) \circ \triangle^r + ( T \otimes \mathrm{id}) \circ \triangle^l) \otimes \mathrm{id}_{M^{\otimes n-i}} ) \circ f \nonumber \\
&+ (-1)^{n+1} ( f \otimes \mathrm{id}) \circ ( \mathrm{id} \otimes T) \circ \triangle ~-~ (-1)^{n+1} ( f \otimes \mathrm{id}) \circ \triangle^l \circ T.  \nonumber
 \end{align}

It follows from (\ref{derived-brkt}) and (\ref{cohoch-rota}) that 
\begin{align*}
d_T f := \llbracket T, f \rrbracket = (-1)^n ~\delta_{\mathrm{coHoch}} (f), ~ \text{ for } f \in \mathrm{Hom}(C, M^{\otimes n}).
\end{align*}
This shows that the cohomology of the relative Rota-Baxter operator $T$ is isomorphisc to the coHochschild cohomology of the coalgebra $M$ with coefficients in the bimodule $C$.

\medskip

Let $T: C \rightarrow M$ be a relative Rota-Baxter operator and consider the corresponding dendriform coalgebra structure on $M$. In the following, we compare the cohomology of the relative Rota-Baxter operator $T$ and the cohomology of the dendriform coalgebra $M$.

For $n \geq 1$, we define maps $\Theta_n : \mathrm{Hom}(C, M^{\otimes n}) \rightarrow \mathrm{Hom} ( \mathbb{K}[C_{n+1}] \otimes M, M^{\otimes n+1})$ by
\begin{align*}
\Theta_n (f) ([r]; ~\_~) = \begin{cases} (-1)^{n+1} ( \mathrm{id} \otimes f) \otimes \triangle^r & r =1 \\
0 & 2 \leq r \leq n \\
(f \otimes \mathrm{id}) \circ \triangle^l & r = n+1.
\end{cases}
\end{align*}

It follows from the dendriform coalgebra structure on $M$ that $\Theta_1 (T)([1]; ~\_~) = \triangle_\prec$ and $\Theta_1 (T) ([2]; ~\_~) = \triangle_\succ$. In other words, $\Theta_1 (T)$ is the multiplication on the operad $\{ \mathrm{Hom}(\mathbb{K}[C_n] \otimes M, M^{\otimes n}) \}_{n \geq 1}$ associated to the dendriform coalgebra structure on $M$. Moreover, the collection $\{ \Theta_n \}$ of maps satisfying the following properties
\begin{align}\label{brkt-pre}
\Theta_{1+n} \llbracket P, Q \rrbracket = [\Theta_1 (P), \Theta_n (Q)],
\end{align}
for $P \in \mathrm{Hom}(C,M)$ and $Q \in \mathrm{Hom}(C, M^{\otimes n})$. Here on the right hand side, we use the degree $-1$ graded Lie bracket on $\oplus_{n \geq 1} \mathrm{Hom}(\mathbb{K}[C_n] \otimes M, M^{\otimes n})$ induced from the operad structure. The verification of (\ref{brkt-pre}) is a straightforward but tedious calculation. See \cite[Lemma 3.4]{das3} for the case of algebras.

As a consequence of this fact, we get the following.

\begin{prop}
The collection $\{ \Theta_n \}$ of maps induces a morphism $\Theta_* : H^\bullet_T (C, M) \rightarrow H^{\bullet +1}_{\mathrm{coDend}}(M,M)$ from the cohomology of a relative Rota-Baxter operator $T : C \rightarrow M$ to the cohomology of the induced dendriform structure on $M$.
\end{prop}

\section{Deformations}\label{sec-4}
In this section, we study the formal deformation theory for dendriform coalgebras. Our main results in this section are analogous to the classical results of deformation theory \cite{gers}.

Let $C = (C, \triangle_\prec, \triangle_\succ)$ be a dendriform coalgebra. Consider the space $C[[t]]$ of formal power series in $t$ with coefficients from $C$. Then $C[[t]]$ is a $\mathbb{K}[[t]]$-module and when $C$ is finite dimensional, we have $C[[t]] = C \otimes_\mathbb{K} \mathbb{K}[[t]].$

\begin{defn} A formal one-parameter deformation of $C$ consists of two formal power series $\triangle_{\prec, t} = \sum_{i \geq 0} \triangle_{\prec, i} ~t^i$ and $\triangle_{\succ, t} = \sum_{i \geq 0} \triangle_{\succ, i}~ t^i$
 in which $\triangle_{\prec, 0 } =  \triangle_\prec$,~  $\triangle_{\succ, 0} = \triangle_\succ$ such that $(C[[t]], \triangle_{\prec, t}, \triangle_{\succ, t})$ is a dendriform coalgebra over $\mathbb{K}[[t]].$
 \end{defn}
 
 Thus, if $(\triangle_{\prec, t} , \triangle_{\succ, t})$ is a deformation of $C$, we must have the following identities hold
 \begin{align}
 \sum_{i+j = n}(\triangle_{\prec, i} \otimes  \text{id}) \circ \triangle_{\prec, j} =~& \sum_{i+j = n} (\text{id} \otimes (\triangle_{\prec, i} + \triangle_{\succ, i})) \circ \triangle_{\prec, j}, \label{d1} \\
\sum_{i+j = n} (\triangle_{\succ, i} \otimes \text{id}) \circ \triangle_{\prec, j} =~& \sum_{i+j = n} (\text{id} \otimes \triangle_{\prec , i}) \circ \triangle_{\succ, j}, \label{d2}\\
\sum_{i+j = n} ((\triangle_{\prec, i} + \triangle_{\succ, i} ) \otimes \text{id} ) \circ \triangle_{\succ, j} =~& \sum_{i+j = n} (\text{id} \otimes \triangle_{\succ, i} ) \circ \triangle_{\succ, j}, \label{d3}
 \end{align}
 for $n \geq 0$.
 These equations are called deformation equations.
 
 To write these system of equations in a compact form, we use the following notations. For each $i \geq 0$, define $\triangle_i \in \mathcal{O}(2) = \text{Hom} (\mathbb{K}[C_2] \otimes C, C^{\otimes 2})$ by
 \begin{align*}
 \triangle_i ([r]; ~\_ ~) = \begin{cases} \triangle_{\prec, i} ~~~ &\text{ if } ~~ [r] = [1] \\ \triangle_{\succ, i} ~~~ &\text{ if } ~~  [r] = [2]. \end{cases}
 \end{align*}
 With these notations, the system of equations (\ref{d1}), (\ref{d2}), (\ref{d3})
 are respectively corresponds to 
 \begin{align*}
 \sum_{i+j=n} (\triangle_i \bullet \triangle_j ) ([r]; ~\_~) = 0, ~~~ \text{ for } [r] = [1], [2], [3].
\end{align*} 
See the appendix for the notation.
Therefore, the deformation equations are equivalently given by $\sum_{i+j = n} \triangle_i \bullet \triangle_j = 0$, for $n \geq 0$. 

It follows from the deformation equations that $\triangle \bullet \triangle_1 + \triangle_1 \bullet \triangle = 0$. This shows that the map $\triangle_1 \in C^2_{\text{coDend}} (C, C) = \text{Hom} (\mathbb{K}[C_2] \otimes C, C^{\otimes 2})$ is a $2$-cocycle, called the infinitesimal of the deformation.

\begin{defn}
Two deformations $(\triangle_{\prec, t}, \triangle_{\succ, t})$ and $(\triangle'_{\prec, t}, \triangle'_{\succ, t})$ of a dendriform coalgebra $C$ are said to be equivalent  if there exists a formal isomorphism $\Phi_t = \sum_{i \geq 0} \Phi_i~t^i : C[[t]] \rightarrow C[[t]]$ with $\Phi_0 = \text{id}$ such that
\begin{align*}
\triangle'_{\prec, t} \circ \Phi_t =~& (\Phi_t \otimes \Phi_t) \circ \triangle_{\prec, t}, \\
\triangle'_{\succ, t} \circ \Phi_t =~& (\Phi_t \otimes \Phi_t) \circ \triangle_{\succ, t}.
\end{align*}
\end{defn}

Note that each $\Phi_i : C \rightarrow C$ can be thought of as an element in $C^1_{\text{coDend}}  (C, C) = \text{Hom} (\mathbb{K}[C_1] \otimes C, C).$ Then the above conditions  of the equivalence can be simply expressed as
\begin{align*}
\sum_{i+j=n} \triangle_i' ([r]; ~\_~) \circ \Phi_j = \sum_{i+j +k = n} ( \Phi_i \otimes \Phi_j ) \circ \triangle_k ([r]; ~\_~),
\end{align*}
for $n \geq 0$ and $[r] = [1], [2].$
The condition for $n = 0$ holds automatically as $\Phi_0 = \text{id}$. For $n = 1$, we have
\begin{align*}
\triangle_1' ([r]; ~\_~) ~+~ \triangle ([r]; ~\_~) \circ \Phi_1 =~ \triangle_1 ([r]; ~\_~) ~+ ~ (\text{id} \otimes \Phi_1 ) \circ \triangle ([r]; ~\_~) ~+~ (\Phi_1 \otimes \text{id}) \circ \triangle ([r]; ~\_~).
\end{align*}
This shows that the difference $\triangle_1' - \triangle_1$ is a coboundary $\delta_c (\Phi_1)$. Therefore, the infinitesimals corresponding to equivalent deformations are cohomologous, hence, they gives rise to same cohomology class in $H^2_{\text{coDend}} (C, C).$ To make a one-to-one correspondence between second cohomology and infinitesimals of equivalence classes of certain deformations, we introduce the following truncated version of formal deformation.

\begin{defn}
An infinitesimal deformation of a dendriform coalgebra $C$ is a deformation of $C$ over $\mathbb{K}[[t]] / (t^2)$ (the local Artinian ring of dual numbers).
\end{defn}

In other words, an infinitesimal deformation of $C$ is given by a pair $(\triangle_{\prec, t}, \triangle_{\succ, t})$ in which $\triangle_{\prec, t} = \triangle_{\prec} + \triangle_{\prec, 1} t$ and $\triangle_{\succ, t} = \triangle_{\succ} + \triangle_{\succ, 1} t$ such that $\triangle_1 = (\triangle_{\prec, 1}, \triangle_{\succ, 1})$ defines a $2$-cocycle in the cohomology of $C$. More precisely, we have the following.

\begin{prop}
There is a one-to-one correspondence between the space of equivalence classes of infinitesimal deformations and the second cohomology group $H^2_{\mathrm{coDend}}(C, C).$\\
\end{prop}

\noindent {\em Rigidity.}

\medskip

In the next, we focus on analytic rigidity of dendriform coalgebras. We start with the following definition of the triviality of a deformation.

\begin{defn}
A deformation $(\triangle_{\prec, t}, \triangle_{\succ, t})$ of a dendriform coalgebra $(C, \triangle_\prec, \triangle_\succ)$ is said to be trivial if it is equivalent to the deformation $(\triangle_{\prec, t}' = \triangle_\prec,~ \triangle_{\succ, t}' = \triangle_\succ).$

\begin{lemma}
Let $(\triangle_{\prec, t}, \triangle_{\succ, t})$ be a non-trivial deformation of a dendriform coalgebra $C$. Then it is equivalent to some deformation $(\triangle_{\prec, t}',~ \triangle_{\succ, t}')$ in which $\triangle_{\prec, t}' = \triangle_\prec + \sum_{i \geq p} \triangle_{\prec, i} t^i$ and $\triangle_{\succ, t}' = \triangle_\succ + \sum_{i \geq p} \triangle_{\succ, i} t^i$, where $\triangle_p' = (\triangle_{\prec, p}', \triangle_{\succ, p}')$ is $2$-cocycle but not a coboundary.
\end{lemma}
\end{defn}

The proof of the above lemma is similar to the standard case of associative (co)algebra \cite{gers, gers-sch}. Hence we omit the details.
As a consequence, we get the following.

\begin{thm}
If $H^2_{\mathrm{coDend}} (C, C) = 0$ then every deformation of $C$ is equivalent to a trivial deformation.
\end{thm}

A dendriform coalgebra $C$ is said to be rigid if every deformation of $C$ is equivalent to a trivial deformation. Thus, the above theorem says that $H^2 = 0$ is a sufficient condition for the rigidity of a dendriform coalgebra.\\

\noindent {\em Extensions of deformations.}

\medskip

A deformation $(\triangle_{\prec, t}, \triangle_{\succ, t})$ of a dendriform coalgebra $C$ is called a deformation of order $n$ if $\triangle_{\prec, t} = \sum_{i=0}^n \triangle_{\prec, i}~ t^i$ and $\triangle_{\succ, t} =  \sum_{i=0}^n \triangle_{\succ, i}~ t^i$. This is equivalent to $\triangle_i = 0$, for $i \geq n+1$.

Suppose there is a map $\triangle_{n+1} = (\triangle_{\prec, n+1},~ \triangle_{\succ, n+1}) \in \text{Hom} (\mathbb{K}[C_2] \otimes C, C^{\otimes 2})$ such that $( \triangle'_{\prec, t} = \triangle_{\prec, t} + \triangle_{\prec, n+1}~ t^{n+1}, \triangle'_{\succ, t} = \triangle_{\succ, t} + \triangle_{\succ, n+1}~ t^{n+1} )$ is a deformation of order $n+1$. Then we say that $(\triangle_{\prec, t}, \triangle_{\succ, t})$ extends to a deformation of order $n+1$. In the following, we consider the problem of extending a deformation of order $n$ to a deformation of order $n+1$.

Since we assume that $(\triangle_{\prec, t}, \triangle_{\succ, t})$ is a deformation of order $n$, we have
\begin{align*}
\triangle \bullet \triangle_i + \triangle_1 \bullet \triangle_{i-1} + \cdots + \triangle_{i-1} \bullet \triangle_1 + \triangle_i \bullet \triangle = 0, ~~~ \text{ for } i= 1, \ldots, n.
\end{align*}
If $(\triangle_{\prec, t}, \triangle_{\succ, t})$ extends to a deformation of order $n+1$, one more deformation equation need to be satisfied:
\begin{align*}
\triangle \bullet \triangle_{n+1} + \triangle_1 \bullet \triangle_{n} + \cdots + \triangle_{n} \bullet \triangle_1 + \triangle_{n+1} \bullet \triangle = 0,
\end{align*}
or, equivalently, $\delta_c (\triangle_{n+1}) = - \sum_{i+j = n+1, i, j \geq 1} \triangle_i \bullet \triangle_j$. The right hand side of this equation is called the obstruction to extend the given deformation.

\begin{lemma}
The obstruction is a $3$-cocycle in the dendriform coalgebra cohomology of $C$ with coefficients in itself.
\end{lemma}

\begin{proof}
We have
\begin{align*}
\delta_{c} \big(- \sum_{i+j =n+1, i, j \geq 1} \triangle_i \bullet \triangle_{j} \big) =~& - \sum_{i+j =n+1, i, j \geq 1} \big( \triangle_i \bullet \delta_{c} (\triangle_j) - \delta_{c} (\pi_i) \bullet \triangle_j \big) \qquad (\text{by } (\ref{mul-circ}))\\
=~& \sum_{p+q+r = n+1, p, q, r \geq 1} \big(  \triangle_p \bullet (\triangle_q \bullet \triangle_r) - (\triangle_p \bullet \triangle_q) \bullet \triangle_r \big)\\ =~& \sum_{p+q+r = n+1, p, q, r \geq 1} A_{p, q, r}    \qquad \mathrm{(say)}.
\end{align*}
The product $\bullet$ is not associative, however, they satisfy the pre-Lie identities (\ref{pre-lie-iden}) (see Appendix). This in particular implies that $A_{p, q, r} = 0$ whenever $q = r$. Finally, if $q \neq r$ then $A_{p, q, r} + A_{p, r, q} = 0$ again by the pre-Lie identity. Hence we have $\sum_{p+q+r = n+1, p, q, r \geq 1} A_{p, q, r} = 0$.
\end{proof}

As a corollary, we obtain the following.
\begin{thm}
If $H^3_{\mathrm{coDend}} (C, C) = 0$ then a deformation of finite order extends to a deformation of next order.
\end{thm}
 
\section{Dendriform coalgebras up to homotopy}\label{sec-5}

The notion of $\text{Dend}_\infty$-algebra is a homotopy version of dendriform algebra and can be thought of like a splitting of $A_\infty$-algebras  \cite{loday-val, das1}. A coderivation interpretation of $\text{Dend}_\infty$-algebra and relation with Rota-Baxter operator on $A_\infty$-algebras has been given in \cite{das1}. Here we introduce $\text{Dend}_\infty$-coalgebras, as homotopy analogue of dendriform coalgebras. We show that they are splitting of $A_\infty$-coalgebras. Further, we introduce Rota-Baxter operator on $A_\infty$-coalgebras and show that they induce $\text{Dend}_\infty$-coalgebras. We start with the definition of $A_\infty$-coalgebras.

\begin{defn} ($A_\infty$-coalgebra) An $A_\infty$-coalgebra consists of a graded vector space $C = \sum C_i$ together with a collection of maps $\triangle_k : C \rightarrow C^{\otimes  k}$ of degree $k-2$, for $k \geq 1$, satisfying the following higher coassociative identities: for all $n \geq 1$,
\begin{align}
\sum_{r+s+t = n,~ r, t \geq  0,~ s \geq 1} ~(-1)^{rs+t}~ (\mathrm{id}^{\otimes r} \otimes \triangle_s \otimes \mathrm{id}^{\otimes t }) \circ \triangle_{r+1+t} = 0.
\end{align}
\end{defn}

$A_\infty$-coalgebras are dual to $A_\infty$-algebras introduced by Stasheff \cite{stas}. It follows from the above definition that an $A_\infty$-coalgebra $(C, \triangle_k)$ with $\triangle_k = 0$, for $k \geq 3$, is a dg-coalgebra. Moreover, if $\triangle_1 = 0$, we obtain a graded coalgebra.

We now introduce Rota-Baxter operator on $A_\infty$-coalgebras.

\begin{defn}\label{rb-inf}
Let $(C, \triangle_k)$ be an $A_\infty$-coalgebra. A Rota-Baxter operator (of weight $0$) on it consists of a degree $0$ map $R : C \rightarrow C$ satisfying
\begin{align}\label{r-b-operator}
\underbrace{(R \otimes \cdots \otimes R)}_{k\text{-times}} \circ \triangle_k =~ (\sum_{i=1}^k (R \otimes \cdots \otimes \underbrace{\text{id}}_{i\text{-th place}} \otimes \cdots \otimes R)) \circ \triangle_k \circ R,
\end{align}
for all $k \geq 1.$
\end{defn}

An $A_\infty$-coalgebra together with a Rota-Baxter operator is called a Rota-Baxter $A_\infty$-coalgebra. This notion can be thought of as homotopy analogue of Rota-Baxter coalgebra.

Now we are ready to define a homotopy version of dendriform coalgebras. The definition is inspired from the definition of $\text{Dend}_\infty$-algebra given in \cite{das1}.

\begin{defn} ($\text{Dend}_\infty$-coalgebra)
A $\text{Dend}_\infty$-coalgebra consists of a graded vector space $C = \oplus C_i$ together with a collection of graded linear maps
\begin{align*}
\{ \triangle_{k, [r]}: C \rightarrow C^{\otimes k} |~ [r] \in C_k, ~ 1 \leq k < \infty \}
\end{align*}
with deg$(\triangle_{k, [r]}) = k-2$, satisfying the following identities
\begin{align}\label{dend-co-inf}
\sum_{r+s+t = n,~ r, t \geq  0,~ s \geq 1} ~(-1)^{rs+t}~ (\mathrm{id}^{\otimes r} \otimes \triangle_{s, R_{r+1} (r+1+t; 1, \ldots, s, \ldots, 1)[\theta]} \otimes \mathrm{id}^{\otimes t }) \circ \triangle_{r+1+t, R_0 (r+1+t; 1, \ldots, s, \ldots, 1)[\theta]} = 0,
\end{align}
for all $n \geq 1$ and $[\theta] \in C_n.$
\end{defn}

Thus, a $\text{Dend}_\infty$-coalgebra is a graded vector space $C$ together with $k$ many $k$-ary cooperations $\{ \triangle_{k, [r]}: C \rightarrow C^{\otimes k} |~ [r] \in C_k \}$ labelled by the elements of $C_k$. These cooperations are supposed to satisfy the identities (\ref{dend-co-inf}).

Therefore, the degree $0$ coproducts $\triangle_\prec := \triangle_{2, [1]}$ and $\triangle_\succ := \triangle_{2, [2]}$ does not satisfy the dendriform coalgebra identities. However, they satisfy the same identities up to some terms involving  $\triangle_{3, [r]}$'s.

An equivalent definition of $\text{Dend}_\infty$-coalgebra can be given by the following. 

\begin{defn} \label{dend-1}
A $\text{Dend}_\infty[1]$-coalgebra consists of a graded vector space $V = \oplus V_i$ together with a collection of degree $-1$ graded linear maps
\begin{align*}
\{ \Delta_{k, [r]}: V \rightarrow V^{\otimes k} |~ [r] \in C_k, ~ 1 \leq k < \infty \}
\end{align*}
satisfying the following identities
\begin{align}
\sum_{r+s+t = n,~ r, t \geq  0,~ s \geq 1} ~(\mathrm{id}^{\otimes r} \otimes \Delta_{s, R_{r+1} (r+1+t; 1, \ldots, s, \ldots, 1)[\theta]} \otimes \mathrm{id}^{\otimes t }) \circ \Delta_{r+1+t, R_0 (r+1+t; 1, \ldots, s, \ldots, 1)[\theta]} = 0,
\end{align}
for all $n \geq 1$ and $[\theta] \in C_n.$
\end{defn}

It is easy to see that if $(A, \triangle_{k, [r]})$ is a $\text{Dend}_\infty$-coalgebra, then $V = s^{-1}A$, the desuspension of $A$, can be equipped with a  $\text{Dend}_\infty [1]$-coalgebra structure and vice versa. The Definition \ref{dend-1} is simple in the sense that all graded maps have same degree and there is no sign after the summation.

Next, we give another equivalent description of a $\text{Dend}_\infty$-coalgebra in terms of a square-zero derivation in a graded diassociative algebra. This is dual to the construction given in \cite{das1}. Let us first recall the following definition from \cite{loday}. 

\begin{defn}A graded diassociative algebra is a graded vector space $D = \oplus D_i$ together with degree zero maps $\dashv, \vdash : D^{\otimes 2} \rightarrow D$ that satisfies
\begin{align*}
& a \dashv (b \dashv c) = (a \dashv b) \dashv c = a \dashv (b \vdash c),\\
& (a \vdash b) \dashv c = a \vdash (b \dashv c),\\
& (a \dashv b) \vdash c = a \vdash (b \vdash c) = (a \vdash b ) \vdash c, ~~~ \text{ for all } a, b , c \in D.
\end{align*}
\end{defn}

A derivation (of degree $d$) in a graded diassociative algebra $(D, \dashv, \vdash)$ is a map $\partial : D \rightarrow D$ of degree $d$ which satisfies
\begin{align*}
\partial \circ \dashv ~=~ \dashv \circ ( \text{id} \otimes \partial + \partial \otimes \text{id} )   ~~~ \text{ and } ~~~ \partial \circ \vdash ~=~ \vdash \circ ( \text{id} \otimes \partial + \partial \otimes \text{id} ).
\end{align*}
We denote the set of all derivations in a graded diassociative algebra by $\text{Der}(D).$

Let $V$ be a graded vector space. Consider the free diassociative algebra $\text{Diass}(V) = TV \otimes V \otimes TV$ over $V$
with coproducts
\begin{align*}
(v_{-n} \cdots v_{-1} \otimes v_0 \otimes v_1 \ldots v_m ) \dashv~& (w_{-p} \cdots w_{-1} \otimes w_{0} \otimes w_1 \cdots w_q ) \\
=~& v_{-n} \cdots v_{-1} \otimes v_0 \otimes v_1 \cdots v_m w_{-p} \cdots w_{-1} w_0 w_1 \cdots w_q,\\
(v_{-n} \cdots v_{-1} \otimes v_0 \otimes v_1 \ldots v_m ) \vdash~& (w_{-p} \cdots w_{-1} \otimes w_{0} \otimes w_1 \cdots w_q ) \\
=~& v_{-n} \cdots v_{-1} v_0 v_1 \ldots v_m w_{-p} \cdots w_{-1} \otimes w_{0} \otimes w_1 \cdots w_q.
\end{align*}
See \cite{loday} for more details. Note that for any $k \geq 1$, there are $k$ many maps from $V^{\otimes k} \rightarrow TV \otimes V \otimes TV$ given by
\begin{align*}
\Pi_1 (v_1 \cdots v_k ) =~& 1 \otimes v_1 \otimes v_2 \cdots v_k,\\
\Pi_2  (v_1 \cdots v_k ) =~& v_1 \otimes v_2 \otimes v_3 \cdots v_k,\\
~& \vdots \\
\Pi_k (v_1 \cdots v_k ) =~& v_1 \cdots v_{k-1} \otimes v_k \otimes 1.
\end{align*}
Thus for any maps $\Delta_{k, [1]}, \ldots, \Delta_{k, [k]} : V \rightarrow V^{\otimes k}$ of degree $-1$, we can define a new map
$\Delta_k : V \rightarrow TV \otimes V \otimes TV $ by $\Delta_k = \Pi_1 \circ \Delta_{k, [1]} + \cdots + \Pi_k \circ \Delta_{k, [k]}$. Then $\Delta_k$ extends to a derivation $\widetilde{\Delta_k} : TV \otimes V \otimes TV \rightarrow TV \otimes V \otimes TV$ on the graded diassociative algebra $\text{Diass}(V) = TV \otimes V \otimes TV$. More precisely, $ \widetilde{\Delta_k}$ is given by
\begin{align*}
\widetilde{\Delta_k} (v_1 \cdots v_n \otimes v_{n+1} \otimes v_{n+2} \cdots v_{n+1+m}) 
= \sum_{i} ~(-1)^{|v_1|+ \cdots + |v_{i-1}|} ~v_1 \cdots v_{i-1} \Delta_k (v_i) v_{i+1} \cdots v_{n+1+m}.
\end{align*}

With these notations, we have the following, the proof of which is dual to \cite[Lemma 4.5]{das1}.
\begin{lemma}\label{lemm}
For any $s , l \geq 1$ with $s + l = n+1$ and for any $1 \leq \theta \leq n$,
the image of $(\widetilde {\Delta_s} \circ \widetilde{\Delta_l} ) ( 1 \otimes v \otimes 1)$ inside $V^{\otimes(\theta -1)} \otimes V \otimes V^{\otimes (n - \theta)}$ is given by
\begin{align*}
 \sum_{r+1+t = l} (\mathrm{id}^{\otimes r} \otimes \Delta_{s, R_{r+1}(r+1+t; 1, \ldots, s, \ldots, 1)[\theta]} \otimes \mathrm{id}^{\otimes t} ) \circ \Delta_{l, R_0 (r+1+t; 1, \ldots, s, \ldots, 1)[\theta]} (v).
\end{align*}
\end{lemma}

Thus we obtain the following interpretation of a $\text{Dend}_\infty$-coalgebra.

\begin{thm}
Let $V$ be a graded vector space and $\{ \Delta_{k, [r]} : V \rightarrow V^{\otimes k } | ~[r] \in C_k, ~ 1 \leq k < \infty \} $ be a sequence of maps of degree $-1$. Consider the derivation $D = \sum_{k \geq 1} \widetilde{\Delta_k} \in \mathrm{Der} (\mathrm{Diass}(V)).$ Then $(V, \Delta_{k, [k]} )$ is a $\mathrm{Dend}_\infty [1]$-coalgebra if and only if $D \circ D = 0$.
\end{thm}

\begin{proof}
The condition $D \circ D = 0$ is equivalent to
\begin{align*}
\widetilde{\Delta_1} \circ \widetilde{\Delta_n} + \widetilde{\Delta_2} \circ \widetilde{\Delta_{n-1}} + \cdots + \widetilde{\Delta_n} \circ \widetilde{\Delta_1} = 0, ~~~ \text{ for all } ~~ n \geq 1.
\end{align*}
This is also same as $ \sum_{s+l = n+1} \widetilde{\Delta_s} \circ \widetilde{\Delta_l} = 0$.
Hence the result follows from Lemma \ref{lemm}.\\
\end{proof}

It has been shown in section \ref{sec-2} that dendriform coalgebras are splitting of associative coalgebras. Here we prove a homotopy version of this result.

\begin{thm}
Let $(C, \triangle_{k, [r]})$ be a $\text{Dend}_\infty$-coalgebra. Then $(C, \triangle_k)$ is an $A_\infty$-coalgebra where
\begin{align*}
\triangle_k = \triangle_{k, [1]} + \cdots + \triangle_{k, [k]}, ~ \text{ for } 1 \leq k < \infty.
\end{align*}
\end{thm}

\begin{proof}
Since $(C, \triangle_{k, [r]})$ is a $\text{Dend}_\infty$-coalgebra, we have (\ref{dend-co-inf}) holds, for all $n \geq 1$ and $[\theta] \in C_n$. Adding up these relations for $[\theta] = [1], \ldots, [n]$, we get
\begin{align}\label{dend-inf-spl}
\sum_{r+s+t = n,~ r, t \geq  0,~ s \geq 1} ~(-1)^{rs+t}~ \sum_{\theta = 1}^n ~(\mathrm{id}^{\otimes r} \otimes \triangle_{s, R_{r+1} (r+1+t; 1, \ldots, s, \ldots, 1)[\theta]} \otimes \mathrm{id}^{\otimes t }) \circ \triangle_{r+1+t, R_0 (r+1+t; 1, \ldots, s, \ldots, 1)[\theta]} = 0.
\end{align}
For any fixed $r, s, t$, we may write
\begin{align*}
\sum_{\theta = 1}^n = \sum_{\theta = 1}^r + \sum_{\theta = r+1}^{r+s} + \sum_{\theta = r+s+1}^{r+s+t}.
\end{align*}
Observe that
\begin{align}\label{eqn-p}
&\sum_{\theta = 1}^r ~(\mathrm{id}^{\otimes r} \otimes \triangle_{s, R_{r+1} (r+1+t; 1, \ldots, s, \ldots, 1)[\theta]} \otimes \mathrm{id}^{\otimes t }) \circ \triangle_{r+1+t, R_0 (r+1+t; 1, \ldots, s, \ldots, 1)[\theta]} \nonumber \\
~&= (\text{id}^{\otimes r} \otimes \triangle_s \otimes \text{id}^{\otimes t} ) \circ \triangle_{r+1+t, [1]+\cdots + [r]}.
\end{align}
Similarly,
\begin{align}\label{eqn-q}
&\sum_{\theta = r+1}^{r+s}~ (\mathrm{id}^{\otimes r} \otimes \triangle_{s, R_{r+1} (r+1+t; 1, \ldots, s, \ldots, 1)[\theta]} \otimes \mathrm{id}^{\otimes t }) \circ \triangle_{r+1+t, R_0 (r+1+t; 1, \ldots, s, \ldots, 1)[\theta]} \nonumber \\
~&= (\text{id}^{\otimes r} \otimes \triangle_{s, [1]+ \cdots + [s]} \otimes \text{id}^{\otimes t} ) \circ \triangle_{r+1+t, [r+1]} \nonumber \\
~&= (\text{id}^{\otimes r} \otimes \triangle_{s} \otimes \text{id}^{\otimes t} ) \circ \triangle_{r+1+t, [r+1]}
\end{align}
and
\begin{align}\label{eqn-r}
&\sum_{\theta = r+s+1}^{r+s+t} ~ (\mathrm{id}^{\otimes r} \otimes \triangle_{s, R_{r+1} (r+1+t; 1, \ldots, s, \ldots, 1)[\theta]} \otimes \mathrm{id}^{\otimes t }) \circ \triangle_{r+1+t, R_0 (r+1+t; 1, \ldots, s, \ldots, 1)[\theta]} \nonumber \\
~&= (\text{id}^{\otimes r} \otimes \triangle_{s} \otimes \text{id}^{\otimes t} ) \circ \triangle_{r+1+t, [r+2] + \cdots + [r+1+t]}.
\end{align}
Hence, by adding (\ref{eqn-p}), (\ref{eqn-q}) and (\ref{eqn-r}), we get
\begin{align*}
\sum_{\theta=1}^n = (\text{id}^{\otimes r} \otimes \triangle_{s} \otimes \text{id}^{\otimes t} ) \circ \triangle_{r+1+t}.
\end{align*}
Therefore, the result follows by substituting this into the identity (\ref{dend-inf-spl}).
\end{proof}

Like Rota-Baxter coalgebras give rise to dendriform coalgebras (Proposition \ref{rota-dend}), Rota-Baxter $A_\infty$-coalgebras give rise to $\text{Dend}_\infty$-coalgebras.

\begin{thm}
Let $(C, \triangle_k)$ be an $A_\infty$-coalgebra and $R$ be a Rota-Baxter operator on it (Definition \ref{rb-inf}). Then the graded vector space $C$ inherits a $\text{Dend}_\infty$-coalgebra structure with coproducts
\begin{align*}
\triangle_{k, [r]} = (R \otimes  \cdots \otimes \underbrace{\mathrm{id}}_{r\mathrm{\text-th~ place}} \otimes \cdots \otimes R) \circ \triangle_k, ~~ \text{ for } k \geq 1 \text{  and }  [r] \in C_k.
\end{align*}

\end{thm}

\begin{proof}
Since $R$ is a Rota-Baxter operator on $(C, \triangle_k)$, we have from (\ref{r-b-operator}) that
\begin{align*}
(R \otimes \cdots \otimes R) \circ \triangle_k = (\sum_{i=1}^k \triangle_{k, [i]}) \circ R.
\end{align*}
Let $1 \leq \theta \leq n$ be fixed. Since $(C, \triangle_k)$ is an $A_\infty$-coalgebra, we have
\begin{align}\label{ro-dend}
\sum_{} (-1)^{rs+t}~ (R \otimes \cdots \otimes \overbrace{\text{id}}^{\theta\text{-th place}} \otimes \cdots \otimes R) \circ (\text{id}^{\otimes r} \otimes \triangle_s \otimes \text{id}^{\otimes t}) \circ \triangle_{r+1+t} = 0.
\end{align}
For any tuple $(r, s, t)$ with $r +s +t = n$, we have either $1 \leq \theta \leq r$ or $r+1 \leq \theta \leq r+s$ or $r+s+1 \leq \theta \leq r+s+t =n$. When $1 \leq \theta \leq r$, we have
\begin{align*}
&(R \otimes \cdots \otimes \overbrace{\text{id}}^{\theta\text{-th place}} \otimes \cdots \otimes R) \circ (\text{id}^{\otimes r} \otimes \triangle_s \otimes \text{id}^{\otimes t}) \circ \triangle_{r+1+t} \\
~&= (R \otimes \cdots \otimes \text{id} \otimes \cdots \otimes R \otimes (\sum_{j=1}^s \triangle_{s, [j]}) \circ R \otimes R \otimes \cdots \otimes R) \circ \triangle_{r+1+t} \\
~&= (\text{id}^{\otimes r} \otimes \sum_{j=1}^s \triangle_{s, [j]} \otimes \text{id}^{\otimes t}) \circ \triangle_{r+1+t, [\theta]} \\
~&= (\text{id}^{\otimes r} \otimes  \triangle_{s, R_{r+1} (r+1+t; 1, \ldots, s , \ldots, 1)[\theta]} \otimes \text{id}^{\otimes t}) \circ \triangle_{r+1+t, R_0 (r+1+t; 1, \ldots, s , \ldots, 1)[\theta]}.
\end{align*}
When $r+1 \leq \theta \leq r+s$, then
\begin{align*}
&(R \otimes \cdots \otimes \overbrace{\text{id}}^{\theta\text{-th place}} \otimes \cdots \otimes R) \circ (\text{id}^{\otimes r} \otimes \triangle_s \otimes \text{id}^{\otimes t}) \circ \triangle_{r+1+t} \\
~&= (R^{\otimes r} \otimes (R \otimes \cdots \otimes \text{id} \otimes \cdots \otimes R) \circ \triangle_s \otimes R^{\otimes t} ) \circ \triangle_{r+1+t} \\
~&= (R^{\otimes r} \otimes \triangle_{s, [\theta -r]} \otimes R^{\otimes t} ) \circ \triangle_{r+1+t} \\
~&= (\text{id}^{\otimes r} \otimes  \triangle_{s, [\theta -r]} \otimes \text{id}^{\otimes t}) \circ \triangle_{r+1+t, [r+1]} \\
~&= (\text{id}^{\otimes r} \otimes  \triangle_{s, R_{r+1} (r+1+t; 1, \ldots, s , \ldots, 1)[\theta]} \otimes \text{id}^{\otimes t}) \circ \triangle_{r+1+t, R_0 (r+1+t; 1, \ldots, s , \ldots, 1)[\theta]}.
\end{align*}
Similarly, if $r+s+1 \leq \theta \leq r+s+t = n$, then we also have
\begin{align*}
&(R \otimes \cdots \otimes \overbrace{\text{id}}^{\theta\text{-th place}} \otimes \cdots \otimes R) \circ (\text{id}^{\otimes r} \otimes \triangle_s \otimes \text{id}^{\otimes t}) \circ \triangle_{r+1+t} \\
~&= (\text{id}^{\otimes r} \otimes  \triangle_{s, R_{r+1} (r+1+t; 1, \ldots, s , \ldots, 1)[\theta]} \otimes \text{id}^{\otimes t}) \circ \triangle_{r+1+t, R_0 (r+1+t; 1, \ldots, s , \ldots, 1)[\theta]}.
\end{align*}
Hence the result follows from (\ref{ro-dend}).
\end{proof}

\section{Appendix}\label{sec-6}

In this appendix, we recall some basics on non-symmetric operads. See \cite{gers-voro, loday-val} for more details.\\

\noindent {\em Non-symmetric operads.}

\medskip

\begin{defn}
A non-symmetric operad in the category of vector spaces consist of a collection of vector spaces $\mathcal{O} = \{ \mathcal{O}(n)|~ n \geq 1 \}$ together with partial compositions
\begin{align*}
\bullet_i : \mathcal{O}(m) \otimes \mathcal{O}(n) \rightarrow \mathcal{O}(m+n-1), ~~~~ 1 \leq i \leq m
\end{align*}
satisfying the following identities
\begin{align*}
(f \bullet_i g) \bullet_{i+j-1} h = f \bullet_i (g \bullet_j h), \quad &\mbox{~~~ for } 1 \leq i \leq m, ~1 \leq j \leq n,\\
(f \bullet_i g) \bullet_{j+n-1} h = (f \bullet_j h) \bullet_i g, \quad  & \mbox{~~~ for } 1 \leq i < j \leq m,
\end{align*}
and there is an (identity) element $\mathrm{id} \in \mathcal{O}(1)$ such that 
$ f \bullet_i \mathrm{id} = f =\mathrm{id} \bullet_1 f,$ for all $f \in \mathcal{O}(n)$ and $1 \leq i \leq n$.
\end{defn}

A toy example of an operad is given by the endomorphism operad $\text{End}_A$ associated to a vector space $A$. One may also consider coendomorphism operad $\text{coEnd}_A$ associated to $A$. More generally, $\text{coEnd}_A (n) := \text{Hom} (A, A^{\otimes n})$, for $n \geq 1$, and the partial compositions are given by 
\begin{align*}f \bullet_i g = (\text{id} \otimes \cdots \otimes \underbrace{\text{g}}_{i\text{-th place}} \otimes \cdots \otimes \text{id}) \circ f.
\end{align*}

In an operad, one may also define a product $\bullet : \mathcal{O}(m) \otimes \mathcal{O}(n) \rightarrow \mathcal{O}(m+n-1)$ by
\begin{align*}
f \bullet g = \sum_{i=1}^m ~(-1)^{(i-1)(n-1)}~ f \bullet_i g, ~~~ f \in \mathcal{O}(m), g \in \mathcal{O}(n).
\end{align*}
The product $\bullet$ is not associative, but, they satisfy the following pre-Lie identities:
\begin{align}\label{pre-lie-iden}
(f \bullet g) \bullet h - f \bullet (g \bullet h) = (-1)^{(n-1)(p-1)} ((f \bullet h) \bullet g - f \bullet (h \bullet g)),
\end{align}
for $f \in \mathcal{O}(m), ~g \in \mathcal{O}(n)$ and $h \in \mathcal{O}(p).$
Therefore, there is a degree $-1$ graded Lie bracket on $\oplus_{n \geq 1} \mathcal{O}(n)$ by
\begin{align}\label{lie-brckt}
[f,g] := f \bullet g - (-1)^{(m-1)(n-1)} g \bullet f.
\end{align}

\begin{defn}
Let $(\mathcal{O}, \bullet_i, \mathrm{id})$ be an operad. A multiplication on it is given by an element $\pi \in \mathcal{O}(2)$ such that $\pi \bullet_1 \pi = \pi \bullet_2 \pi$, or, equivalently, 
$\pi \bullet \pi = 0$.
\end{defn}

A multiplication $\pi$ defines an associative product $\oplus_{n \geq 1} \mathcal{O}(n)$ by $f \cdot g = \pm~ (\pi \bullet_2 g) \bullet_1 f$ and a degree $+1$ differential $d_\pi : \mathcal{O}(n) \rightarrow \mathcal{O}(n+1)$ by $d_\pi f = \pi \bullet f - (-1)^{k-1} f \bullet \pi$. Then 
for any $f, g \in \mathcal{O}(2)$, it is easy to see that
\begin{align}\label{mul-circ}
\delta_\pi (f \bullet g) =  f \bullet \delta_{\pi} (g) - \delta_{\pi} (f) \bullet g ~+~ g \cdot f -~ f \cdot g.
\end{align}
It was shown in \cite{gers-voro} that the bracket $[~, ~]$ and the product $\cdot$ induces a Gerstenhaber algebra structure on the graded space of cohomology $H^\bullet (\mathcal{O}, d_\pi).$

\vspace{0.5cm}

\noindent {\bf Acknowledgements.} The author would like to thank Indian Institute of Technology (IIT), Kanpur for financial support.



\end{document}